\documentclass[a4paper,12pt,reqno]{amsart}
\usepackage{a4wide}
\usepackage{amsmath}
\usepackage{amssymb}
\usepackage{amsthm}
\usepackage{latexsym}
\usepackage{graphicx}
\usepackage[english]{babel}
              
\newtheorem{prop}[subsection]{Proposition}
\newtheorem{conj}[subsection]{Conjecture}
\newtheorem{teor}[subsection]{Theorem}
\newtheorem{lema}[subsection]{Lemma}
\newtheorem{cor} [subsection]{Corollary}
\theoremstyle{definition}

\theoremstyle{remark}
\newtheorem{obs} [subsection]{Remark}
\newtheorem{exm} [subsection]{Example}

\def\hdepth{\operatorname{hdepth}}

\numberwithin{equation}{section}

\begin{document}

\title[On the Hilbert depth of a special class of squarefree monomial ideals]
      {On the Hilbert depth of a special class of squarefree monomial ideals}
\author[Andreea I.\ Bordianu, Mircea Cimpoea\c s]{Andreea I.\ Bordianu$^1$, Mircea Cimpoea\c s$^2$}  
\date{}



\footnotetext[1]{ \emph{Andreea I.\ Bordianu}, National University of Science and Technology Politehnica Bucharest, Faculty of
Applied Sciences, 
Bucharest, 060042, E-mail: andreea.bordianu@stud.fsa.upb.ro}
\footnotetext[2]{ \emph{Mircea Cimpoea\c s}, National University of Science and Technology Politehnica Bucharest, Faculty of
Applied Sciences, 
Bucharest, 060042, Romania and Simion Stoilow Institute of Mathematics, Research unit 5, P.O.Box 1-764,
Bucharest 014700, Romania, E-mail: mircea.cimpoeas@upb.ro,\;mircea.cimpoeas@imar.ro}

\begin{abstract}
Let $r$ and $n$ be two positive integers and $S=K[x_1,\ldots,x_{n+r-1}]$, the ring of polynomials in $n+r-1$ variable, over a field $K$.
We consider the squarefree monomial ideal $I_{n,r}:= x_1 \cdots x_{r-1} (x_{r},\ldots,x_{r+n-1}) \subset S$ and
we prove several results regarding the Hilbert depth of $S/I_{n,r}$. Also, we consider the special case $n=r$.
\end{abstract}

\maketitle

\section{Introduction}

Let $K$ be a field and let $S=K[x_1,\ldots,x_N]$ be the polynomial ring over $K$, in $N$ variables.
Let $M$ be a finitely generated graded $S$-module. The Hilbert depth of $M$, denoted by $\hdepth(M)$, is the 
maximal depth of a finitely generated graded $S$-module $U$ with the same Hilbert series as $M$; see
\cite{bruns,maxim,uli} for further details.

Let $0\subset I\subsetneq J\subset S$ be two squarefree monomial ideals. For $0\leq j\leq n$, let $\alpha_j(J/I)$ be
the number of squarefree monomials $u\in J\setminus I$ of degree $j$.
For all $0\leq k\leq  q\leq N$, we consider the integers:
\begin{equation}\label{betak}
  \beta_k^q(J/I):=\sum_{j=0}^k (-1)^{k-j} \binom{q-j}{k-j} \alpha_j(J/I).
\end{equation}
In \cite[Theorem 2.4]{lucrare2} it was proved that
\begin{equation}\label{hdep}
\hdepth(J/I)=\max\{q\;:\;\beta_k^q(J/I)\geq 0\text{ for all }0\leq k\leq q\}.
\end{equation}
Using this formula, our goal is to study the Hilbert depth of $S/I_{n,r}$, where
$$I_{n,r}=x_1\cdots x_{r-1}(x_r,\ldots,x_{n+r-1})\subset S=K[x_1,\ldots,x_{n+r-1}].$$
In Proposition \ref{pro1}, we note that $\hdepth(I_{n,r})=r+\left\lfloor \frac{n-1}{2} \right\rfloor$.

In Corollary \ref{cc1}, we show that, if $n\geq 2$ then, $\hdepth(S/I_{n,r})\leq n+r-3-a$, where $a$ is the maximal 
integer with $\binom{a+r}{r}<n$. In Theorem \ref{t0}, we prove that if $n > r\geq 2$ then 
$\hdepth(S/I_{n,r})\leq n+r-\left\lfloor \frac{n-r-1}{3} \right\rfloor -3$.

In Theorem \ref{t1}, we show that 
$\hdepth(S/I_{n,r})\geq \left\lfloor \frac{n}{2} \right\rfloor + r - 2.$
In Theorem \ref{t2}, we prove that if $n\geq 2$ and $r\geq 1 + \binom{n-2}{\left\lfloor \frac{n-2}{2} \right\rfloor}$, then
$\hdepth(S/I_{n,r})=n+r-3$. Consequently, we show that the class of ideals $I_{n,r}$ provides examples of monomial ideals
$I$ with $\hdepth(S/I)=\hdepth(I)+k$, for any $k\geq 0$; see Remark \ref{ob1}. Also, from Theorem \ref{t2} we deduce immediately
that $\lim\limits_{r\to\infty} \frac{\hdepth(S/I_{n,r})}{r} = 1,\text{ for all }n\geq 2$; see Corollary \ref{ct1}.

In Theorem \ref{t4}, we prove that, for any $\varepsilon>0$, there exists a constant $A=A(\varepsilon,r)$ such that
$\hdepth(S/I_{n,r})\leq \left\lceil \frac{n}{2} \right\rceil + (r-1)\left\lfloor \varepsilon n \right\rfloor + A + r - 2,\text{ for all }n\geq 2$. As a consequence of Theorem \ref{t1} and Theorem \ref{t4}, we prove that 
$\lim\limits_{n\to\infty} \frac{\hdepth(S/I_{n,r})}{n} = \frac{1}{2},\text{ for all }r\geq 2$; see Corollary \ref{ct2}.

In Theorem \ref{t3}, we prove that $\left\lfloor \frac{11n}{6} \right\rfloor - 2 \geq \hdepth(S/I_{n,n})\geq \left\lfloor \frac{9n}{5} \right\rfloor - 2$. Also, we propose some very tight bounds $\hdepth(S/I_{n,n})$, in Conjecture \ref{conj}.

\section{Main results}

Let $n,r\geq 1$ be two integers. Let $N=n+r-1$ and $S=K[x_1,\ldots,x_N]$. We consider the ideal
$$I_{n,r}:=x_1x_2\cdots x_{r-1}(x_{r},x_{r+1},\cdots,x_{r+n-1})\subset S.$$

\begin{prop}\label{pro1}
With the above notation, we have that $\hdepth(I_{n,r})=r+\left\lfloor \frac{n-1}{2} \right\rfloor$.
\end{prop}

\begin{proof}
According to \cite[Lemma 2.10]{lucrare2} and \cite[Lemma 2.13]{lucrare2}, it follows that
$$\hdepth(I_{n,r})=r-1+\hdepth((x_{r},x_{r+1},\cdots,x_{r+n-1})K[x_r,\ldots,x_{r+n-1}]).$$
The conclusion follows from \cite[Example 3.4]{uli}.
\end{proof}

\begin{lema}\label{lemon}
With the above notation, we have that $\alpha_j(I_{n,r})=\begin{cases} 0,& j\leq r-1 \\ \binom{n}{j-r+1},& r\leq j\leq N \end{cases}$.
\end{lema}

\begin{proof}
Since $I_{n,r}$ is generated in degree $r$, it follows that $\alpha_j(I_{n,r})=0$, for $0\leq j\leq r-1$.
It is clear that any squarefree monomial $u\in I_{n,r}$ of degree $j\geq r$ is of the form $u=x_1x_2\cdots x_{r-1}v$, where $v\in K[x_r,\ldots,x_{n+r-1}]$ 
is a squarefree monomial of degree $j-r+1$. It follows that, for $j\geq r$, $\alpha_j(I_{n,r})$ is equal to the number of squarefree monomials of degree $j-r+1$ in $n$ variables. Hence,
we get the required result.
\end{proof}

We recall the following useful lemma:

\begin{lema}\label{magic}(\cite[Lemma 2.4]{lucrare3})
For any integers $0\leq k\leq q$ and $n\geq 0$, we have that:
$$ \sum_{j=0}^k (-1)^{k-j} \binom{q-j}{k-j}\binom{n}{j} = (-1)^k \binom{q-n}{k} = \binom{n-q+k-1}{k}.$$
\end{lema}

\begin{prop}\label{p1}
Let $0\leq k\leq q\leq N$. Then:
\begin{enumerate}
\item[(1)] $\beta_k^q(I_{n,r})=0$, for $0 \leq k\leq r-1$.
\item[(2)] $\beta_k^q(I_{n,r})=\binom{n-q+k-1}{k-r+1} - (-1)^{k-r+1} \binom{q-r+1}{k-r+1}$, for $r \leq k\leq q$.
\item[(3)] $\beta_k^q(S/I_{n,r})=\binom{N-q+k-1}{k}$, for $0 \leq k\leq r-1$.
\item[(4)] $\beta_k^q(S/I_{n,r})=\binom{N-q+k-1}{k} - \binom{n-q+k-1}{k-r+1} - (-1)^{k-r} \binom{q-r+1}{k-r+1}$, for $r \leq k\leq q$.
\end{enumerate}
\end{prop}

\begin{proof}
(1) According to Lemma \ref{lemon}, $\alpha_j(I_{n,r})=0$ for $j\leq r-1$. Hence, the conclusion follows from \eqref{betak}.

(2) According to Lemma \ref{lemon} and \eqref{betak}, we have that:
    $$\beta_k^q(I_{n,r})=\sum_{j=r}^k (-1)^{k-j} \binom{q-j}{k-j} \binom{n}{j-r+1}.$$
		Using the substitution $\ell=j-r+1$, we get:
		\begin{align*}
		\beta_k^q(I_{n,r}) & =\sum_{\ell=1}^{k-r+1} (-1)^{k-r+1-\ell} \binom{q-r+1-\ell}{k-r+1-\ell} \binom{n}{\ell} = \\
		                   & =\sum_{\ell=0}^{k-r+1} (-1)^{k-r+1-\ell} \binom{q-r+1-\ell}{k-r+1-\ell} \binom{n}{\ell} -
											    (-1)^{k-r+1}\binom{q-r+1}{k-r+1}.
		\end{align*}
		Using Lemma \ref{magic}, we get the required result.
		
(3,4) Since $\alpha_j(S/I_{n,r})=\binom{N}{j}-\alpha_j(I_{n,r})$, for all $0\leq j\leq N$, from Lemma \ref{magic} it follows that
      $$\beta_k^q(S/I_{n,r})=\binom{N-q+k-1}{k} -\beta_k^q(I_{n,r}),\text{ for all }0\leq k\leq q\leq N.$$
			Hence, the conclusion follows from (1) and (2).
\end{proof}

From Proposition \ref{p1}(3) it follows that 
$$\beta_k^{r-1}(S/I_{n,r}) = \binom{N-r+k}{k}=\binom{n-1+k}{k}\geq 0,\text{ for all }0\leq k\leq r-1,$$
and thus $\hdepth(S/I_{n,r})\geq r-1$. Moreover, from Proposition \ref{p1}(3,4) and \eqref{hdep} we deduce:

\begin{prop}\label{c1}
With the above notation, we have that:
\begin{align*}
\hdepth(S/I_{n,r}) =  \max\{& r-1\leq q\leq n+r-2 \;:\; \binom{N-q+k-1}{k} - \binom{n-q+k-1}{k-r+1} \geq \\ & \geq \binom{q-r+1}{k-r+1}, 
 \text{ for all } r\leq k\leq q\text{ with }k-r\text{ even.}\}. 
\end{align*}
\end{prop}


\begin{cor}\label{cc1}
With the above notation, we have that:
$$\hdepth(S/I_{n,r}) \leq \max\{ r-1\leq q\leq n+r-2 \;:\; \binom{n+2r-q-2}{r} \geq n \}.$$
Moreover, if $n\geq 2$ and $a=\max\{j\geq 0\;:\;\binom{j+r}{r}<n\}$, then $\hdepth(S/I_{n,r})\leq n+r-3-a$.
\end{cor}

\begin{proof}
Let $q=\hdepth(S/I_{n,r})$. According to Corollary \ref{c1}, if we take $k=r$, then
$$\binom{N-q+r-1}{r}-\binom{n-q+r-1}{1} \geq \binom{q-r+1}{1},$$
which is equivalent to $\binom{n+2r-q-2}{r} \geq n$. On the other hand, if $\binom{n+2r-q'-2}{r}<n$ for some $q'\leq n+r-1$,
then $$\binom{N-q'+r-1}{r}-\binom{n-q'+r-1}{1} < \binom{q'-r+1}{1},$$
and therefore, according Proposition \ref{c1}, we have $q'>q$.

Now, assume $n\geq 2$, let $a=\max\{j\geq 0\;:\;\binom{j+r}{r}<n\}$ and $q'=n+r-2-a$. It follows that
$$\binom{n+2r-q'-2}{r}=\binom{a+r}{r}<n,$$
and thus $\hdepth(S/I_{n,r})\leq q'-1 =n+r-3-a$, as required.
\end{proof}

\begin{exm}\rm
(1) If $n=1$, then $I_{1,r}=(x_1\cdots x_r)\subset S=K[x_1,\ldots,x_r]$ is a principal ideal, and therefore $\hdepth(S/I_{1,r})=r-1$
and $\hdepth(I_{1,r})=r$; see \cite[Theorem 2.3]{bordi1}.

(2) If $r=1$, then $I_{n,1}=(x_1,\ldots,x_n)\subset S=K[x_1,\ldots,x_n]$ is the maximal graded ideal of $S$. In particular, we
have $\hdepth(S/I_{n,1})=0$ and $\hdepth(I_{n,1})=\left\lceil \frac{n}{2} \right\rceil$; see \cite[Example 3.4]{uli} and \cite{maxim}.

(3) If $r=2$, then $I_{n,2}$ is the edge ideal of a star graph. In particular, according to \cite[Theorem 2.5]{cipu}, we have
$\hdepth(S/I_{n,2})\geq \left\lfloor \frac{n}{2} \right\rfloor + \left\lfloor \sqrt{n} \right\rfloor - 2,$
while, according to Proposition \ref{pro1}, we have $\hdepth(I_{n,2})=\left\lfloor \frac{n+3}{2} \right\rfloor$.
\end{exm}

In the following, we will tacitly assume that $n,r\geq 2$.

\begin{teor}\label{t0}
If $n > r\geq 2$ then $\hdepth(S/I_{n,r})\leq n+r-\left\lfloor \frac{n-r-1}{3} \right\rfloor -3$.
\end{teor}

\begin{proof}
Let $N=n+r-1$ and $q=N-\left\lfloor \frac{n-r-1}{3} \right\rfloor -1$. Let $T=\left\lfloor \frac{n-r-1}{3} \right\rfloor$ and 
$k=r+2\left\lfloor \frac{T}{2} \right\rfloor$. Note that \small
$$q-k = n - \left\lfloor \frac{n-r-1}{3} \right\rfloor - 2 - 2\left\lfloor \frac{T}{2} \right\rfloor \geq
        n - 2\left\lfloor \frac{n-r-1}{3} \right\rfloor - 2 \geq n - \frac{2n-2r-2}{3} - 2 = \frac{n+2r-4}{3} \geq 0,$$ \normalsize
and thus $r\leq k\leq q$. Also, $k-r=2s$, where $s=\left\lfloor \frac{T}{2} \right\rfloor$.
We have 
\begin{equation}\label{ecoo1}
\binom{N-q+k-1}{k}=\binom{N-q+k-1}{N-q-1}=\binom{ \left\lfloor \frac{n-r-1}{3} \right\rfloor +r + 2\left\lfloor \frac{T}{2} \right\rfloor }{\left\lfloor \frac{n-r-1}{3} \right\rfloor} \leq \binom{ 2\left\lfloor \frac{n-r-1}{3} \right\rfloor +r}{ \left\lfloor \frac{n-r-1}{3} \right\rfloor}
\end{equation}
Also, since $n-q+k-1=N-q+k-r$ and $(n-q+k-1)-(k-r+1)=N-q-1$, we have:
\begin{equation}\label{ecoo2}
\binom{n-q+k-1}{k-r+1}
=\binom{N-q+k-r}{N-q+1}=\binom{ \left\lfloor \frac{n-r-1}{3} \right\rfloor +1 + 2\left\lfloor \frac{T}{2} \right\rfloor }{\left\lfloor \frac{n-r-1}{3} \right\rfloor}\geq 1
\end{equation}
We claim that
\begin{equation}\label{cleim}
n-\left\lfloor \frac{n-r-1}{3} \right\rfloor - 1\geq 2\left\lfloor \frac{n-r-1}{3} \right\rfloor + 2.
\end{equation}
Indeed, \eqref{cleim} is equivalent to $n-3\geq 3\left\lfloor \frac{n-r-1}{3} \right\rfloor$, which is true,
since $3\left\lfloor \frac{n-r-1}{3} \right\rfloor\leq n-r-1\leq n-3$, as $r\geq 2$. Using \eqref{cleim} and the fact that
$2\left\lfloor \frac{T}{2} \right\rfloor+1 \in \{\left\lfloor \frac{n-r-1}{3} \right\rfloor,\left\lfloor \frac{n-r-1}{3} \right\rfloor+1 \}$, it follow that
\begin{equation}\label{ecoo3}
\binom{q-r+1}{k-r+1}=\binom{n-\left\lfloor \frac{n-r-1}{3} \right\rfloor-1}{2\left\lfloor \frac{T}{2} \right\rfloor+1}
\geq \binom{n-\left\lfloor \frac{n-r-1}{3} \right\rfloor-1}{\left\lfloor \frac{n-r-1}{3} \right\rfloor}.
\end{equation}
Since $n-\left\lfloor \frac{n-r-1}{3} \right\rfloor-1\geq 2 \left\lfloor \frac{n-r-1}{3} \right\rfloor + r$, from \eqref{ecoo1} and
\eqref{ecoo3} it follows that
$$\binom{q-r+1}{k-r+1} \geq \binom{n-\left\lfloor \frac{n-r-1}{3} \right\rfloor-1}{\left\lfloor \frac{n-r-1}{3} \right\rfloor}
\geq \binom{ 2\left\lfloor \frac{n-r-1}{3} \right\rfloor +r}{ \left\lfloor \frac{n-r-1}{3} \right\rfloor} \geq \binom{N-q+k-1}{k}.$$
Since, by \eqref{ecoo2}, $\binom{n-q+k-1}{k-r+1}\geq 1$, it follows that
$$\binom{N-q+k-1}{k} - \binom{n-q+k-1}{k-r+1} < \binom{q-r+1}{k-r+1},$$
and thus, by Proposition \ref{c1}, we get 
$$\hdepth(S/I_{n,r})\leq q-1=n+r-\left\lfloor \frac{n-r-1}{3} \right\rfloor -3,$$
as required.
\end{proof}

\begin{teor}\label{t1}
With the above notation, we have that: 
$$\hdepth(S/I_{n,r})\geq \left\lfloor \frac{n}{2} \right\rfloor + r - 2.$$
\end{teor}

\begin{proof}
Since 
$\binom{N-q+k-1}{k}=\frac{(n+r-q+k-2)(n+r-q+k-3)\cdots (n-q+k)}{k(k-1)\cdots (k-r+2)}\cdot \binom{n-q+k-1}{k-r+1},$
from Proposition \ref{c1} it follows that $\hdepth(S/I_n)$ is the maximal value of $q\geq r-1$, with $q\leq n+r-1$, such that
\begin{equation}\label{31}
\frac{(n+r-q+k-2)\cdots (n-q+k)}{k(k-1)\cdots (k-r+2)}\geq 1 + \frac{(q-r+1)(q-r)\cdots (q-k+1)}{(n-q+k-1)\cdots (n-q+r-1)},
\end{equation}
for all $r\leq k\leq q$ with $k-r$ even.

Let $q=\left\lfloor \frac{n}{2} \right\rfloor + r - 2$. In order to prove the result, according to \eqref{31}, it
is enough to show that
$$ \frac{(\left\lceil \frac{n}{2} \right\rceil + k) \cdots (\left\lceil \frac{n}{2} \right\rceil + k-r+2) }{k(k-1)\cdots (k-r+2)} \geq 1 +
   \frac{(\left\lfloor \frac{n}{2} \right\rfloor-1)(\left\lfloor \frac{n}{2} \right\rfloor-2)\cdots (\left\lfloor \frac{n}{2} \right\rfloor+r-k-1)}
	{(\left\lceil \frac{n}{2} \right\rceil + k-r+1)(\left\lceil \frac{n}{2} \right\rceil + k-r)\cdots(\left\lceil \frac{n}{2} \right\rceil+1)},$$
	for all $r\leq k\leq q$ with $k-r$ even. For $n\in\{2,3\}$, there is nothing to prove, so we may assume $n\geq 4$.
	Since the right side term is less than $2$, it is enough to prove that:
	$$ \frac{(\left\lceil \frac{n}{2} \right\rceil + k) \cdots (\left\lceil \frac{n}{2} \right\rceil + k-r+2) }{k(k-1)\cdots (k-r+2)} \geq 2,\text{ for all }
	   r\leq k\leq q\text{ with }k-r\text{ even }.$$
	Since $\frac{a+1}{b+1} < \frac{a}{b}$ for $a>b>0$, it is enough to prove the above inequality for $k=q=\left\lfloor \frac{n}{2} \right\rfloor + r - 2$, that is
	$$ \frac{(n+r-2) \cdots (n+1)n }{(\left\lfloor \frac{n}{2} \right\rfloor + r - 2) \cdots (\left\lfloor \frac{n}{2} \right\rfloor +1)\left\lfloor \frac{n}{2} \right\rfloor } \geq 2, $$
	which is obviously true, as $r\geq 2$. Hence, the proof is complete.
\end{proof}

\begin{teor}\label{t2}
If $n\geq 2$ and $r\geq 1 + \binom{n-2}{\left\lfloor \frac{n-2}{2} \right\rfloor}$ then
$$\hdepth(S/I_{n,r})=n+r-3.$$
\end{teor}

\begin{proof}
Let $q:=n+r-3$. According to Proposition \ref{c1}, it is enough to show that
$$\binom{k+1}{k} - \binom{k-r+2}{k-r+1} \geq \binom{n-2}{k-r+1},$$
for all $k$ such that $0\leq k-r\leq q-r=n-3$ and $k-r$ is even. The above inequality
is equivalent to $$ r\geq 1 + \binom{n-2}{k-r+1},\text{ for all }0\leq k-r\leq n-3\text{ with }k-r\text{ even,}$$
which is true, by hypothesis.
\end{proof}

\begin{cor}\label{ct1}
With the above notation, we have that: $$\lim_{r\to\infty} \frac{\hdepth(S/I_{n,r})}{r} = 1,\text{ for all }n\geq 1.$$
\end{cor}

\begin{proof}
It follows immediately from Theorem \ref{t2}.
\end{proof}

\begin{obs}\label{ob1}\rm
Since $n\geq 2$, the ideal $I_{n,r}\subset S=K[x_1,\ldots,x_{n+r-1}]$ is not principal, and therefore, by \cite[Theorem 2.2]{bordi1},
 $\hdepth(S/I_{n,r})\leq n+r-3$. Theorem \ref{t2} shows that, for $r\geq 1 + \binom{n-2}{\left\lfloor \frac{n-2}{2} \right\rfloor}$, $\hdepth(S/I_{n,r})$ reaches its maximal (possible) value, with respect to the number of variables of $S$.
On the other hand, according to Proposition \ref{pro1}, we have $\hdepth(I_{n,r})=r+\left\lfloor \frac{n-1}{2} \right\rfloor$.
It follows that $$\hdepth(S/I_{n,r})-\hdepth(I_{n,r}) = \left\lfloor \frac{n}{2} \right\rfloor - 2,\text{ for all }n\geq 2\text{ and }
r\geq 1 + \binom{n-2}{\left\lfloor \frac{n-2}{2} \right\rfloor}.$$
This shows, in particular, that for any $k\geq 0$, we can find a polynomial ring $S$ and a squarefree monomial ideal $I\subset S$
such that $\hdepth(S/I)=\hdepth(I)+k$.
\end{obs}

The following result is a generalization of \cite[Theorem 2.6]{cipu}.

\begin{teor}\label{t4}
Let $n,r\geq 2$. Let $\varepsilon>0$ and $A:=\left\lfloor\frac{1}{8\varepsilon}\log \left(  \frac{1}{2} + \frac{1}{4(r-1)\varepsilon}  \right)\right\rfloor+r-1$. Then:
$$\hdepth(S/I_{n,r})\leq \left\lceil \frac{n}{2} \right\rceil + (r-1)\left\lfloor \varepsilon n \right\rfloor + A + r - 2,\text{ for all }n\geq 2.$$
\end{teor}

\begin{proof}
Let $q:=\left\lceil \frac{n}{2} \right\rceil + (r-1)\left\lfloor \varepsilon n \right\rfloor + A + r - 1$. Let $n_0$ be the smallest
$n$ for which $q\leq n+r-2$. If $n<n_0$ then the assertion is trivial, since $\hdepth(S/I_{n,r})\leq n+r-3$,
as $I_{n,r}$ is not pricipal. Assume $n\geq n_0$. 

Since $\binom{n-q+2s+r-1}{2s+1}>0$, according to Proposition \ref{c1}, it is enough to find $0\leq s\leq \left\lfloor \frac{q}{2} \right\rfloor$ such that
\begin{equation}\label{dorinta1}
\binom{N-q+2s+r-1}{2s+r} \leq \binom{q-r+1}{2s+1},
\end{equation}
Note that, \eqref{dorinta1} is equivalent to 
\begin{equation}\label{dorinta2}
\binom{\left\lfloor \frac{n}{2} \right\rfloor - (r-1)\left\lfloor n\varepsilon \right\rfloor-A+2s-1}{2s+r} \leq 
\binom{\left\lceil \frac{n}{2} \right\rceil + (r-1)\left\lfloor n\varepsilon  \right\rfloor + A}{2s+1}.
\end{equation}
If $n=2p$ then \eqref{dorinta2} is equivalent to 
\begin{equation}\label{dorinta3}
\binom{p - (r-1)\left\lfloor 2p\varepsilon \right\rfloor-A+2s-1}{2s+r} \leq 
\binom{p + (r-1)\left\lfloor 2p \varepsilon  \right\rfloor + A}{2s+1}.
\end{equation}
If $n=2p+1$ then \eqref{dorinta2} is equivalent to 
\begin{equation}\label{dorinta4}
\binom{p - (r-1)\left\lfloor (2p+1)\varepsilon \right\rfloor-A+2s-1}{2s+r} \leq 
\binom{p +1+ (r-1)\left\lfloor (2p+1) \varepsilon \right\rfloor + A}{2s+1}.
\end{equation}
Note that, if \eqref{dorinta3} holds then so does \eqref{dorinta4}. Thus, in order to complete the proof,
it is enough to find $0\leq s\leq \left\lfloor \frac{q}{2} \right\rfloor$ such that 
\eqref{dorinta3} holds.

Let $x:=p-(r-1)\left\lfloor 2p\varepsilon \right\rfloor-A$. Note that \eqref{dorinta3} is equivalent to
$$\frac{(x+2s-1)(x+2s-2)\cdots (x-r)}{(2s+r)(2s+r-1)\cdots (2s+2)} \leq (x+2(r-1) \left\lfloor 2p\varepsilon \right\rfloor + 2A) \cdots 
 (x+2(r-1) \left\lfloor 2p\varepsilon \right\rfloor + 2A-2s),$$
which is again equivalent to \small
$$\frac{(x+2s-1)\cdots (x+2s-r+1)}{(2s+r)\cdots (2s+2)} \leq 
\frac{(x+2(r-1) \left\lfloor 2p\varepsilon \right\rfloor + 2A) \cdots 
 (x+2(r-1) \left\lfloor 2p\varepsilon \right\rfloor + 2A-2s)}{(x+2s-r)\cdots (x-r)},$$ \normalsize
which is again equivalent to
$$\prod_{j=0}^{r-2} \left( 1+ \frac{x-r-1}{2s+2+j} \right) \leq \prod_{j=0}^{2s} \left( 1 + \frac{2(r-1)\left\lfloor 2p\varepsilon \right\rfloor + 2A +r-2s }{x-r+j} \right).$$
In order to prove the above inequality, it suffices to show that
$$(r-1)\log\left( 1 + \frac{x-r-1}{2s+2} \right) \leq \sum_{j=0}^{2s} \log\left( 1 + \frac{2(r-1)\left\lfloor 2p\varepsilon \right\rfloor + 2A +r-2s }{x-r+j} \right).$$
We claim that $s=(r-1)(\left\lfloor 2p\varepsilon \right\rfloor+1)$ does the job. In other words, we claim that  \small
\begin{equation}\label{dorintaa}
(r-1)\log\left ( 1 + \frac{p-(r-1)\left\lfloor 2p\varepsilon \right\rfloor-A-r-1}
{2(r-1) (\left\lfloor 2p\varepsilon \right\rfloor+1) + 2} \right)
\leq \sum_{j=0}^{2(r-1)(\left\lfloor 2p\varepsilon \right\rfloor+1)} 
\log\left( 1 + \frac{  2A - r + 2 }{p-(r-1)\left\lfloor 2p\varepsilon \right\rfloor-A-r+j} \right).
\end{equation} \normalsize
Since $ \left\lfloor 2p\varepsilon \right\rfloor+1 > 2p\varepsilon $, it follows that
$$ 1 + \frac{p-(r-1)\left\lfloor 2p\varepsilon \right\rfloor-A-r-1}{2(r-1) (\left\lfloor 2p\varepsilon \right\rfloor+1) + 2} =
\frac{1}{2} +  \frac{p-A-2}{2(r-1) (\left\lfloor 2p\varepsilon \right\rfloor+1) + 2} < \frac{1}{2} + \frac{1}{4(r-1)\varepsilon}.$$
Let $B:= \log \left(  \frac{1}{2} + \frac{1}{4(r-1)\varepsilon}  \right)$.
Since $\log(1+y)<\frac{y}{y+1}$, in order to prove \eqref{dorintaa}, we need to show that 
$$\sum_{j=0}^{2(r-1)(\left\lfloor 2p\varepsilon \right\rfloor+1)} 
\frac{2A - r +2 }{p-(r-1)\left\lfloor 2p\varepsilon \right\rfloor + A - 2r +2 + j }\geq (r-1)B.$$
Note that
$$\sum_{\ell=-(r-1)(\left\lfloor 2p\varepsilon \right\rfloor+1)}^{(r-1)(\left\lfloor 2p\varepsilon \right\rfloor+1)} 
\frac{2A - r +2 }{p + A - r +1 + \ell }  > \frac{2(r-1)(\left\lfloor 2p\varepsilon \right\rfloor+1)(2A-r+2)}{p+A-r+1}.$$
Hence, it suffices to show that 
\begin{equation}\label{dorintaa2}
\frac{2(\left\lfloor 2p\varepsilon \right\rfloor+1)(2A-r+2)}{p+A-r+1} >B.
\end{equation}
Since $\left\lfloor 2p\varepsilon \right\rfloor+1 > 2p\varepsilon$, in order to prove \eqref{dorintaa2}, it is enough to show that
$$ A\geq r-1\text{ and }4\varepsilon (2A-r+2) >B.$$
Since $A=\left\lfloor \frac{B}{8\varepsilon} \right\rfloor + r-1$, the above inequalities are satisfied. Hence, the proof is complete.
\end{proof}

\begin{cor}\label{ct2}
With the above notation, we have that: $$\lim_{n\to\infty} \frac{\hdepth(S/I_{n,r})}{n} = \frac{1}{2},\text{ for all }r\geq 2.$$
\end{cor}

\begin{proof}
According to Theorem \ref{t1}, we have that
\begin{equation}\label{estim1}
\frac{\hdepth(S/I_{n,r})}{n} \geq \frac{1}{2} + \frac{r-3}{n},\text{ for all }n\geq 2.
\end{equation}
On the other hand, according to Theorem \ref{t4}, for any $\varepsilon>0$, there exists a constant $A=A(\varepsilon,r)$ such that
\begin{equation}\label{estim2}
\frac{\hdepth(S/I_{n,r})}{n} \leq \frac{1}{2} + (r-1)\varepsilon + \frac{A+r-1}{n},\text{ for all }n\geq 2.
\end{equation}
Now, let $\delta>0$ and put $\epsilon=\frac{\delta}{r}$. Then, there exists $n_{\delta}>0$ such that $\frac{r-3}{n}<\delta$
and $\frac{A+r-1}{n}<\epsilon$. From \eqref{estim1} and \eqref{estim2} it follows that
$$\left|\frac{\hdepth(S/I_{n,r})}{n} - \frac{1}{2} \right| <\delta,\text{ for all }n\geq n_{\delta}.$$
Hence, we get the required result.
\end{proof}

\pagebreak

\section{The case $n=r$.}

We consider the ideal:
$$I_{n,n}=(x_1\cdots x_n,x_1\cdots x_{n-1}x_{n+1},\ldots,x_1\cdots x_{n-1}x_{2n-1})\subset S=K[x_1,\ldots,x_{2n-1}].$$
Note that $I_{n,n}$ is minimally generated by $n$ monomials of degree $n$. Our aim is to find upper and lower bounds 
for $\hdepth(S/I_{n,n})$.

First, we prove the following technical lemmas:

\begin{lema}\label{tehno}
For any $n\geq 2$, we have that: $$\binom{\left\lfloor \frac{5n}{6} \right\rfloor }{2\left\lfloor \frac{n}{6} \right\rfloor + 1 }\geq 
\binom{\left\lfloor \frac{3n}{2} \right\rfloor}{\left\lfloor \frac{n}{6} \right\rfloor}.$$
\end{lema}

\begin{proof}
If $2\leq n\leq 5$ then $\left\lfloor \frac{n}{6} \right\rfloor=0$ and the conclusion is trivial, as $\left\lfloor \frac{5n}{6} \right\rfloor\geq 1$.
Assume $n\geq 6$. We consider several cases:
\begin{enumerate}
\item[(1)] If $n=6p+1$, then the conclusion is equivalent to 
\begin{equation}\label{ek2}
\binom{5p}{2p+1}\geq \binom{9p+1}{p},\text{ for all }p\geq 1.
\end{equation}
Equation \eqref{ek2} is equivalent to
\begin{equation}\label{ekk2}
5p(5p-1)\cdots (3p) \geq (9p+1)(9p)\cdots (8p+2)(2p+1)(2p)\cdots (p+1).
\end{equation}
Since $3p\geq 2p+1$, in order to prove \eqref{ekk2}, it is enough to show
\begin{equation}\label{ekkk2}
5p(5p-1)\cdots (3p+1) \geq (9p+1)(9p)\cdots (8p+2)(2p)\cdots (p+1).
\end{equation}
For $0\leq j\leq p-1$, we have that \small
$$(5p-j)(4p-j) = 20p^2-9pj+j^2 \geq (18p^2+2p)-(11p+1)j+j^2=(9p+1-j)(2p-j),$$ \normalsize
from which, equation \eqref{ekkk2} follows. Hence, \eqref{ek2} holds.

\item[(2)] If $n=6p+2$, then the conclusion is equivalent to 
\begin{equation}\label{ek3}
\binom{5p+1}{2p+1}\geq \binom{9p+3}{p},\text{ for all }p\geq 1.
\end{equation}
Equation \eqref{ek3} is equivalent to
\begin{equation}\label{ekk3}
(5p+1)(5p)\cdots (3p+1) \geq (9p+3)(9p)\cdots (8p+4)(2p+1)(2p)\cdots (p+1).
\end{equation}
Since $3p+1\geq 2p+1$, in order to prove \eqref{ekk3}, it is enough to show
\begin{equation}\label{ekkk3}
(5p+1)(5p)\cdots (3p+2) \geq (9p+3)(9p)\cdots (8p+4)(2p)\cdots (p+1).
\end{equation}
For $0\leq j\leq p-1$, we have that \small
$$(5p+1-j)(4p+1-j) = 20p^2+9p+1-(9p+2)j+j^2 \geq (18p^2+6p)-(11p+3)j+j^2=(9p+3-j)(2p-1),$$ \normalsize
from which, equation \eqref{ekkk3} follows. Hence, \eqref{ek3} holds.

\item[(3)] If $n=6p+3$, then the conclusion is equivalent to 
\begin{equation}\label{ek4}
\binom{5p+2}{2p+1}\geq \binom{9p+4}{p},\text{ for all }p\geq 1.
\end{equation}
Since $\frac{\binom{5p+2}{2p+1}}{\binom{5p+1}{2p+1}}=\frac{5p+2}{3p+1}> \frac{9p+4}{8p+4} = \frac{\binom{9p+4}{p}}{\binom{9p+3}{p}},$
\eqref{ek4} follows from \eqref{ek3}.

\item[(4)] If $n=6p+4$, then the conclusion is equivalent to 
\begin{equation}\label{ek5}
\binom{5p+3}{2p+1}\geq \binom{9p+6}{p},\text{ for all }p\geq 1.
\end{equation}
Equation \eqref{ek5} is equivalent to
\begin{equation}\label{ekk5}
(5p+3)(5p+2)\cdots (3p+3) \geq (9p+6)(9p+5)\cdots (8p+7)(2p+1)(2p)\cdots (p+1).
\end{equation}
Since $3p+3\geq 2p+1$, in order to prove \eqref{ekk5}, it is enough to show
\begin{equation}\label{ekkk5}
(5p+3)(5p+2)\cdots (3p+4) \geq (9p+6)(9p+5)\cdots (8p+7)(2p)\cdots (p+1).
\end{equation}
For $0\leq j\leq p-1$, we have that \small
$$(5p+3-j)(4p+3-j) = 20p^2+27p+9-(9p+6)j+j^2 \geq (18p^2+12p)-(11p+6)j+j^2=(9p+6-j)(2p-1),$$ \normalsize
from which, equation \eqref{ekkk5} follows. Hence, \eqref{ek5} holds.

\item[(5)] If $n=6p+5$, then the conclusion is equivalent to 
\begin{equation}\label{ek6}
\binom{5p+4}{2p+1}\geq \binom{9p+7}{p},\text{ for all }p\geq 1.
\end{equation}
Since $\frac{\binom{5p+4}{2p+1}}{\binom{5p+1}{2p+3}}=\frac{5p+4}{3p+3}> \frac{9p+7}{8p+7} = \frac{\binom{9p+7}{p}}{\binom{9p+6}{p}},$
\eqref{ek6} follows from \eqref{ek5}.

\item[(6)] If $n=6p$, then the conclusion is equivalent to 
\begin{equation}\label{ek1}
\binom{5p}{2p+1}\geq \binom{9p}{p},\text{ for all }p\geq 1.
\end{equation}
Note that \eqref{ek1} follows from \eqref{ek2}.
\end{enumerate}
\end{proof}

\begin{lema}\label{tech}
We have that:
\begin{enumerate}
\item[(1)] $\binom{6p+2s}{p}\geq \binom{p+2s+1}{2s+1}+\binom{4p-1}{2s+1}$, for all $p\geq 1$ and $0\leq s\leq 2p-1$.
\item[(2)] $\binom{6p+2s+2}{p+1}\geq \binom{p+2s+2}{2s+1}+\binom{4p-1}{2s+1}$, for all $p\geq 1$ and $0\leq s\leq 2p-1$.
\item[(3)] $\binom{6p+2s+3}{p+1}\geq \binom{p+2s+2}{2s+1}+\binom{4p}{2s+1}$, for all $p\geq 1$ and $0\leq s\leq 2p-1$.
\item[(4)] $\binom{6p+2s+4}{p+1}\geq \binom{p+2s+2}{2s+1}+\binom{4p+1}{2s+1}$, for all $p\geq 1$ and $0\leq s\leq 2p$.
\item[(5)] $\binom{6p+2s+5}{p+1}\geq \binom{p+2s+2}{2s+1}+\binom{4p+2}{2s+1}$, for all $p\geq 1$ and $0\leq s\leq 2p$.
\end{enumerate}
\end{lema}

\begin{proof}
(1) Note that $p+2s+1 > 4p-1$ if and only if $s \geq \left\lfloor \frac{3p}{2} \right\rfloor$. Thus,
    in order to prove (1) for $\left\lfloor \frac{3p}{2} \right\rfloor \leq s\leq 2p-1$, it is enough to show that 
		$$\binom{6p+2s}{p}\geq 2\binom{p+2s+1}{2s+1}=2\binom{p+2s+1}{p}.$$
    This inequality is trivial, since $6p+2s\geq 2(p+2s+1)$ for $s\leq 2p-1$. Hence, we may 
     assume that $0\leq s\leq \left\lfloor \frac{3p}{2} \right\rfloor-1$. 
		
		In order to complete the proof of (1), 
		it is enough to show that 
		\begin{equation}\label{wish}
		\binom{6p+2s}{p}\geq 2\binom{4p-1}{2s+1},
		\end{equation}
		for all $0\leq s\leq \left\lfloor \frac{3p}{2} \right\rfloor-1$. If $s\geq p-1$ then
		$$ \binom{6p+2s}{p} \geq \binom{8p-2}{p}\text{ and }\binom{4p-1}{2s+1}\leq \binom{4p-2}{2p-1}.$$		
		Hence, in order to prove \eqref{wish} for $p-1\leq s\leq \left\lfloor \frac{3p}{2} \right\rfloor-1$,
		it is enough to show \eqref{wish} for $s=p-1$, that is 
		$$\binom{8p-2}{p}\geq 2\binom{4p-1}{2p-1}=\binom{4p}{2p}.$$
		In fact, we will show something more, namely:
		\begin{equation}\label{strict}
		\binom{7p-1}{p}\geq \binom{4p}{2p}.
		\end{equation}
		Indeed \eqref{strict} is equivalent to
		$$ \frac{(7p-1)(7p-1)\cdots (6p)}{p(p-1)\cdots 1} \geq \frac{4p(4p-1)\cdots(2p+1)}{(2p)(2p-1)\cdots 1},$$
		which is again equivalent to
		\begin{equation}\label{stricte}
		(7p-1)\cdots (6p)(2p)\cdots (p+1)\geq (4p)\cdots (2p+1).
		\end{equation}
		For $0\leq j\leq p-1$ we have that
		$$(6p+j)(p+1+j)=(6p^2+6p)+(7p+1)j+j^2 \geq (6p^2+5p+1)+(5p+2)j+j^2=(3p+1+j)(2p+1+j).$$
		From this, \eqref{stricte} follows immediately. Thus \eqref{strict} holds.
		On the other hand, if $\frac{p-1}{2}\leq s\leq p-1$ then 
		$$\binom{6p+2s}{p}\geq \binom{7p-1}{p}\text{ and }2\binom{4p-1}{2s+1}\leq 2\binom{4p-1}{2p-1}=\binom{4p}{2p}.$$
		Hence, according to the previous arguments, it follows that \eqref{wish} holds for $s$ with $\frac{p-1}{2}\leq s \leq \left\lfloor \frac{3p}{2} \right\rfloor-1$.
		
		In order to complete the proof of (1), it suffice to show that
		\begin{equation}\label{ultima}
		\binom{6p+j}{p} \leq 2\binom{4p-1}{j+1},\text{ for }0\leq j\leq p-1.
		\end{equation}
		We prove \eqref{ultima} by induction on $j\leq p-1$.
		From \eqref{strict}, it follows that \eqref{ultima} holds for $j=p-1$.
		Now, let $0\leq j\leq p-2$ and assume that \eqref{ultima} holds for $j+1$.
		We have $$\frac{\binom{6p+j}{p}}{\binom{6p+j+1}{p}}=\frac{5p+j+1}{6p+j+1}\text{ and }\frac{\binom{4p-1}{j+1}}{\binom{4p-1}{j+2}}=\frac{j+1}{4p-j-1}.$$
		Since $0\leq j\leq p-2$, it is clear that $\frac{5p+j+1}{6p+j+1}\geq \frac{j+1}{4p-j-1}$. From induction hypothesis, it follows that \small
		$$\binom{6p+j}{p} = \frac{5p+j+1}{6p+j+1}\binom{6p+j+1}{p}\geq \frac{5p+j+1}{6p+j+1}\binom{4p-1}{j+2} \geq \frac{j+1}{4p-j-1}\binom{4p-1}{j+2}
		=\binom{4p-1}{j+1},$$ \normalsize
		and thus, the induction step is completed.

(2) Note that $p+2s+2 > 4p-1$ if and only if  $s \geq \left\lfloor \frac{3p-1}{2} \right\rfloor$. Thus,
    in order to prove (2) for $\left\lfloor \frac{3p-1}{2} \right\rfloor \leq s\leq 2p-1$, it is enough to show that 
		$$\binom{6p+2s+2}{p+1}\geq 2\binom{p+2s+2}{2s+1}=2\binom{p+2s+2}{p+1},$$
		which is trivial, since $6p+2s+2 \geq 2(p+2s+2)$ for $s\leq 2p-1$.
		
		Now, assume $s\leq \left\lfloor \frac{3p-1}{2} \right\rfloor-1$. In order to complete the proof of (2), 
		it is enough to show that 
		\begin{equation}\label{wish2}
		\binom{6p+2s+2}{p+1}\geq 2\binom{4p-1}{2s+1},
		\end{equation}
		for all $0\leq s\leq \left\lfloor \frac{3p-1}{2} \right\rfloor-1$. Since $\binom{6p+2s+2}{p+1}\geq \binom{6p+2s}{p}$, 
		the conclusion follows from \eqref{wish}.

(3) We note that $p+2s+2 > 4p$ if and only if $s \geq \left\lfloor \frac{3p}{2} \right\rfloor$.
    Thus, in order to prove (3) for $\left\lfloor \frac{3p}{2} \right\rfloor \leq s\leq 2p-1$, it is enough to show that 
		$$\binom{6p+2s+3}{p+1}\geq 2\binom{p+2s+2}{2s+1}=2\binom{p+2s+2}{p+1},$$
    which is trivial, since $6p+2s+3\geq 2(p+2s+2)$ for $s\leq 2p-1$.

		Now, assume $s\leq \left\lfloor \frac{3p}{2} \right\rfloor-1$. In order to complete the proof of (3), 
		it is enough to show that 
		\begin{equation}\label{wish3}
		\binom{6p+2s+3}{p+1}\geq 2\binom{4p}{2s+1},
		\end{equation}
		for all $0\leq s\leq \left\lfloor \frac{3p}{2} \right\rfloor-1$. 
		We claim that 
		\begin{equation}\label{strict3}
		\binom{7p+1}{p+1}\geq 2\binom{4p}{2p+1}.
		\end{equation}
		Indeed, \eqref{strict3} is equivalent to
		$$\frac{(7p+1)(7p)\cdots (6p+1)}{(p+1)!} \geq \frac{(4p)(4p-1)\cdots (2p)}{(2p+1)!},$$
		which is further equivalent to
		\begin{equation}\label{stricte3}
		(7p+1)(7p)\cdots (6p+1)(2p+1)\cdots (p+2) \geq 2(4p)(4p-1) \cdots (2p)
		\end{equation}
		Since $6p+1>2(3p)=6p$, in order to prove \eqref{stricte3} it is enough to show that
		\begin{equation}\label{sstricte3}
		(7p+1)(7p)\cdots (6p+2)(2p+1)\cdots (p+2) \geq (4p)\cdots (3p+1)(3p-1)\cdots (2p).
		\end{equation}		
		For $0\leq j\leq p-1$ we have that
		\begin{align*}
		(6p+2+j)(p+2+j)&=(6p^2+14p+4)+(7p+4)j+j^2 \geq \\
		& \geq (6p^2+2p)+(5p+1)j+j^2=(3p+1+j)(2p+j).
		\end{align*}
		From this, \eqref{sstricte3} follows immediately. Thus \eqref{strict3} holds.
		
		As in the proof of (1), from \eqref{strict3} it follows that \eqref{wish3} holds for $s\geq \frac{p-1}{2}$.
		In order to complete the proof of (3), it suffice to show that
		\begin{equation}\label{ultima3}
		\binom{6p+j+3}{p+1} \leq 2\binom{4p}{j+1},\text{ for }0\leq j\leq p-1.
		\end{equation}
    We use induction of $j\leq p-1$. The case $j=p-1$ follows from \eqref{strict3}. If $j\leq p-2$ and \eqref{ultima3} holds for $j+1$,
		in order to prove \eqref{ultima3} for $j$, it is enough to notice that
		$$\frac{\binom{6p+j+3}{p+1}}{\binom{6p+j+4}{p+1}}=\frac{5p+j+3}{6p+j+4}\geq \frac{j+2}{4p-j-1} = \frac{2\binom{4p}{j+1}}{2\binom{4p}{j+2}}.$$
		Hence, the proof is complete.
		
(4) We note that $p+2s+2 > 4p+1$ if and only if $s \geq \left\lfloor \frac{3p+1}{2} \right\rfloor$.
    Thus, in order to prove (4) for $\left\lfloor \frac{3p+1}{2} \right\rfloor \leq s\leq 2p$, it is enough to show that 
		$$\binom{6p+2s+4}{p+1}\geq 2\binom{p+2s+2}{2s+1}=2\binom{p+2s+2}{p+1},$$
    which is trivial, since $6p+2s+4\geq 2(p+2s+2)$ for $s\leq 2p$.
		
		Now, assume $s\leq \left\lfloor \frac{3p+1}{2} \right\rfloor-1$. In order to complete the proof of (4), 
		it is enough to show that 
		\begin{equation}\label{wish4}
		\binom{6p+2s+4}{p+1}\geq 2\binom{4p+1}{2s+1},
		\end{equation}
		for all $0\leq s\leq \left\lfloor \frac{3p+1}{2} \right\rfloor-1$. 
		We claim that 
		\begin{equation}\label{strict4}
		\binom{7p+2}{p+1}\geq 2\binom{4p+1}{2p+1}.
		\end{equation}
		Indeed, \eqref{strict4} is equivalent to
		$$\frac{(7p+2)(7p+1)\cdots (6p+2)}{(p+1)!} \geq \frac{2(4p+1)(4p)\cdots (2p+1)}{(2p+1)!},$$
		which is further equivalent to
		\begin{equation}\label{stricte4}
		(7p+2)(7p+1)\cdots (6p+2)(2p+1)\cdots (p+2) \geq 2(4p+1)(4p)\cdots (2p+1).
		\end{equation}
		Simplifying with $(6p+2)$ and $2(3p+1)$, respectively, in \eqref{stricte4}, we get the 
		equivalent inequality
   \begin{equation}\label{sstricte4}
		(7p+2)(7p+1)\cdots (6p+3)(2p+1)\cdots (p+2) \geq (4p+1)\cdots(3p+2)(3p)\cdots(2p+1).
		\end{equation}
		For $0\leq j\leq p-1$ we have that
		\begin{align*}
		(6p+3+j)(p+2+j)&=(6p^2+15p+6)+(7p+5)j+j^2 \geq \\
		& \geq (6p^2+7p+2)+(5p+3)j+j^2=(3p+2+j)(2p+1+j).
		\end{align*}
		From this, \eqref{sstricte4} follows immediately. Thus \eqref{strict4} holds.
		
		As in the proof of (1), from \eqref{strict4} it follows that \eqref{wish4} holds for $s\geq \frac{p-1}{2}$.
		In order to complete the proof of (4), it suffice to show that
		\begin{equation}\label{ultima4}
		\binom{6p+j+4}{p+1} \leq 2\binom{4p+1}{j+1},\text{ for }0\leq j\leq p-1.
		\end{equation}
    We use induction of $j\leq p-1$. The case $j=p-1$ follows from \eqref{strict3}. If $j\leq p-2$ and \eqref{ultima4} holds for $j+1$,
		in order to prove \eqref{ultima4} for $j$, it is enough to notice that
		$$\frac{\binom{6p+j+4}{p+1}}{\binom{6p+j+5}{p+1}}= \frac{5p+j+4}{6p+j+5} \geq \frac{j+2}{4p-j} = \frac{2\binom{4p+1}{j+1}}{2\binom{4p+1}{j+2}}.$$
		Hence, the proof is complete.

(5)	We note that $p+2s+2 > 4p+2$ if and only if $s \geq \left\lfloor \frac{3p}{2} \right\rfloor+1$.
    Thus, in order to prove (5) for $\left\lfloor \frac{3p}{2} \right\rfloor \leq s\leq 2p$, it is enough to show that 
		$$\binom{6p+2s+5}{p+1}\geq 2\binom{p+2s+2}{2s+1}=2\binom{p+2s+2}{p+1},$$
    which is trivial, since $6p+2s+5\geq 2(p+2s+2)$ for $s\leq 2p$.
		
		Now, assume $s\leq \left\lfloor \frac{3p}{2} \right\rfloor$. In order to complete the proof of (5), 
		it is enough to show that 
		\begin{equation}\label{wish5}
		\binom{6p+2s+5}{p+1}\geq 2\binom{4p+2}{2s+1},
		\end{equation}
		for all $0\leq s\leq \left\lfloor \frac{3p}{2} \right\rfloor$. 
	  We claim that 
		\begin{equation}\label{strict5}
		\binom{7p+3}{p+1}\geq 2\binom{4p+2}{2p+1}.
		\end{equation}
		The assertion holds for $p=1$, so we may assume that $p\geq 2$.		
		Note that \eqref{strict5} is equivalent to
		$$\frac{(7p+3)(7p+2)\cdots (6p+3)}{(p+1)!} \geq \frac{2(4p+2)(4p+1)\cdots (2p+2)}{(2p+1)!},$$
		which is further equivalent to
		\begin{equation}\label{stricte5}
		(7p+2)(7p+1)\cdots (6p+2)(2p+1)\cdots (p+2) \geq 2(4p+2)(4p+1)\cdots (2p+2).
		\end{equation}
		Since $p\geq 2$, it follows that $7p+2\geq 2(3p+2)=6p+4$, hence, in order to prove \eqref{stricte4},
		it is enough to show
		\begin{equation}\label{sstricte5}
		(7p+1)\cdots (6p+2)(2p+1)\cdots (p+2) \geq (4p+2)\cdot (3p+3)(3p+1)\cdots (2p+2).
		\end{equation}
		For $0\leq j\leq p-1$ we have that
		\begin{align*}
		(6p+2+j)(p+2+j)&=(6p^2+14p+4)+(7p+4)j+j^2 \geq \\
		& \geq (6p^2+12p+5)+(5p+6)j+j^2=(3p+3+j)(2p+2+j).
		\end{align*}
		From this, \eqref{sstricte5} follows immediately. Thus \eqref{strict5} holds.
		
		As in the proof of (1), from \eqref{strict5} it follows that \eqref{wish5} holds for $s\geq \frac{p-1}{2}$.
		In order to complete the proof of (5), it suffice to show that
		\begin{equation}\label{ultima5}
		\binom{6p+j+5}{p+1} \leq 2\binom{4p+2}{j+1},\text{ for }0\leq j\leq p-1.
		\end{equation}
    We use induction of $j\leq p-1$. The case $j=p-1$ follows from \eqref{strict5}. If $j\leq p-2$ and \eqref{ultima5} holds for $j+1$,
		in order to prove \eqref{ultima5} for $j$, it is enough to notice that
		$$\frac{\binom{6p+j+5}{p+1}}{\binom{6p+j+6}{p+1}}= \frac{5p+j+5}{6p+j+6} \geq \frac{j+2}{4p+1-j} = \frac{2\binom{4p+2}{j+1}}{2\binom{4p+2}{j+2}}.$$
		Hence, the proof is complete.
\end{proof}


\begin{teor}\label{t3}
With the above notation, for all $n\geq 2$, we have that:
$$\left\lfloor \frac{11n}{6} \right\rfloor - 2 \geq \hdepth(S/I_{n,n})\geq \left\lfloor \frac{9n}{5} \right\rfloor - 2.$$
\end{teor}

\begin{proof}
\textbf{The first inequality.} Let $q=\left\lfloor \frac{11n}{6} \right\rfloor - 1$ and $N=2n-1$. According to Proposition \ref{c1}, in order to prove the inequality,
we need to find some $n\leq k\leq q$, with $k-n$ even, such that 
\begin{equation}\label{voim}
\binom{N-q+k-1}{k}-\binom{n-q+k-1}{k-n+1} < \binom{q-n+1}{k-n+1}.
\end{equation}
We choose $k=n+2\left\lfloor \frac{n}{6} \right\rfloor$ and we note that \eqref{voim} is equivalent to
\begin{equation}\label{voim2}
\binom{2n- \left\lfloor \frac{5n}{6} \right\rfloor + 2 \left\lfloor \frac{n}{6} \right\rfloor - 1 }{ n - \left\lfloor \frac{5n}{6} \right\rfloor - 1 }-
\binom{n- \left\lfloor \frac{5n}{6} \right\rfloor + 2 \left\lfloor \frac{n}{6} \right\rfloor}{ n - \left\lfloor \frac{5n}{6} \right\rfloor - 1 }
 < \binom{\left\lfloor \frac{5n}{6} \right\rfloor}{2\left\lfloor \frac{n}{6} \right\rfloor+1}.
\end{equation}
Note that 
\begin{align*}
& 2n- \left\lfloor \frac{5n}{6} \right\rfloor + 2 \left\lfloor \frac{n}{6} \right\rfloor - 1 < 2n - \frac{5n}{6} + 2\left\lfloor \frac{n}{6} \right\rfloor
\leq 2n- \frac{5n}{6} + \frac{2n}{6} = \frac{3n}{2}\\
& n - \left\lfloor \frac{5n}{6} \right\rfloor - 1 < n-\frac{5n}{6} = \frac{n}{6}\text{ and }\left\lfloor \frac{3n}{2} \right\rfloor > 2\left\lfloor \frac{n}{6} \right\rfloor
\end{align*}
Therefore, in order to prove \eqref{voim2}, it suffices to show
$$\binom{\left\lfloor \frac{5n}{6} \right\rfloor}{2\left\lfloor \frac{n}{6} \right\rfloor+1} \geq \binom{\left\lfloor \frac{3n}{2} \right\rfloor}{\left\lfloor \frac{n}{6} \right\rfloor}.$$
This follows from Lemma \ref{tehno}, so we are done.

\textbf{The second inequality.} For $n\geq 4$, we prove it by case verification, so we may assume that $n\geq 5$.
Let $q=\left\lfloor \frac{9n}{5} \right\rfloor - 2$ and $N=2n-1$. According to Proposition \ref{c1}, in order to prove the inequality, 
we have to show that
\begin{equation}\label{vrem}
\binom{N-q+k-1}{k}-\binom{n-q+k-1}{k-n+1}\geq \binom{q-n+1}{k-n+1},
\end{equation}
for all $k$ with $0\leq k-n \leq q-n$ and $k-n=2s$. By straightforward computations, we note that \eqref{vrem} is equivalent to
\begin{equation}\label{vrem2}
\binom{2n-\left\lfloor \frac{4n}{5} \right\rfloor+2s}{n-\left\lfloor \frac{4n}{5} \right\rfloor} \geq 
\binom{n-\left\lfloor \frac{4n}{5} \right\rfloor+2s+1}{2s+1}+\binom{\left\lfloor \frac{4n}{5} \right\rfloor-1}{2s+1},
\end{equation}
for all $s$ with $0\leq 2s \leq \left\lfloor \frac{4n}{5} \right\rfloor-2$. We consider several cases:
\begin{itemize}
\item If $n=5p$ with $p\geq 1$, then \eqref{vrem2} is equivalent to
$$\binom{6p+2s}{p}\geq \binom{p+2s+1}{2s+1}+\binom{4p-1}{2s+1},\text{ for all }0\leq s\leq 2p-1.$$
Hence, we are done by Lemma \ref{tech}(1).
\item If $n=5p+1$ with $p\geq 1$, then \eqref{vrem2} is equivalent to
$$\binom{6p+2s+2}{p+1}\geq \binom{p+2s+2}{2s+1}+\binom{4p-1}{2s+1},\text{ for all }0\leq s\leq 2p-1.$$
Hence, we are done by Lemma \ref{tech}(2).
\item If $n=5p+2$ with $p\geq 1$, then \eqref{vrem2} is equivalent to
$$\binom{6p+2s+3}{p+1}\geq \binom{p+2s+2}{2s+1}+\binom{4p}{2s+1},\text{ for all }0\leq s\leq 2p-1.$$
Hence, we are done by Lemma \ref{tech}(3).
\item If $n=5p+3$ with $p\geq 1$, then \eqref{vrem2} is equivalent to
$$\binom{6p+2s+4}{p+1}\geq \binom{p+2s+2}{2s+1}+\binom{4p+1}{2s+1},\text{ for all }0\leq s\leq 2p.$$
Hence, we are done by Lemma \ref{tech}(4).
\item If $n=5p+4$ with $p\geq 1$, then \eqref{vrem2} is equivalent to
$$\binom{6p+2s+5}{p+1}\geq \binom{p+2s+2}{2s+1}+\binom{4p+2}{2s+1},\text{ for all }0\leq s\leq 2p.$$
Hence, we are done by Lemma \ref{tech}(5).
\end{itemize}
Now, the proof is complete.
\end{proof}

Based on our computer experiments, we propose the following:

\begin{conj}\label{conj}
There
exists a constant $\alpha\approx 1.817$ such that
$$\hdepth(S/I_{n,n})\in \{ \left\lfloor \alpha n \right\rfloor - 2, \left\lfloor \alpha n \right\rfloor - 1 \},\text{ for all }n\geq 2.$$
\end{conj}

\begin{obs}\rm
In fact, we can conjecture something more general. For any $t\in [0,1]$, there exists a constant $\alpha(t)\in [\frac{1}{2},1]$ such that
$$\lim_{n\to\infty} \frac{1}{n}\hdepth(S/I_{\left\lfloor tn \right\rfloor+n_0, \left\lfloor (1-t)n \right\rfloor+r_0}) = \alpha(t),$$
where $n_0,r_0\geq 2$ are some integers. Note that, according to Corollary \ref{ct1}, we have that
$$\alpha(0)=\lim_{n\to\infty} \frac{1}{n}\hdepth(S/I_{n_0,r_0+n}) = 1.$$
Also, according to Corollary \ref{ct2}, we have that
$$\alpha(1)=\lim_{n\to\infty} \frac{1}{n}\hdepth(S/I_{n+n_0,r_0}) = \frac{1}{2}.$$
Moreover, we believe that the map $t\mapsto \alpha(t)$ is decreasing and analytic.
\end{obs}

\bigskip


\end{document}